\documentclass[12pt,a4paper]{article}

\usepackage{amsmath}
\usepackage{amssymb}
\usepackage{graphicx}
\usepackage{amsthm}
\usepackage{amsfonts}
\usepackage{amscd}
\usepackage{cite}

\newtheorem{thm}{Theorem}[section]

\newtheorem{example}[thm]{Example}
\newtheorem{theorem}[thm]{Theorem}

\newtheorem{lemma}[thm]{Lemma}

\newtheorem{definition}[thm]{Definition}

\newtheorem{remark}[thm]{Remark}
\newtheorem{fact}[thm]{Fact}

\begin{document}
\title
{
Perfect Multi Deletion Codes
Achieve
the Asymptotic Optimality of Code Size
}

\author{
Takehiko MORI
\thanks{
Department of Mathematics and Informatics,
Graduate School of Science and Engineering,
Chiba University
1-33 Yayoi-cho, Inage-ku, Chiba City,
Chiba Pref., JAPAN, 263-0022
}
 \and
Manabu HAGIWARA
\thanks{
Department of Mathematics and Informatics,
Graduate School of Science,
Chiba University
1-33 Yayoi-cho, Inage-ku, Chiba City,
Chiba Pref., JAPAN, 263-0022
}
}

\date{}
\maketitle

\begin{abstract}
This paper studies on the cardinality of perfect multi deletion binary codes.
The lower bound for any perfect deletion code with the fixed code length and the number of deletions, and
the asymptotic achievable of Levenshtein's upper bound are shown.
\end{abstract}

\section{Introduction}\label{sec:1}
Thanks to Levenshtein's pioneering work in the 1960's,
research on deletion codes has continued to be an attractive field in
information theory for over 50 years \cite{levenshtein1966binary}.
He found that VT codes were capable of correcting single deletion errors,
generalized VT codes for correcting single insertion/deletion/substitution errors,
and invented a decoding algorithm for VT codes.
He also provided an asymptotic lower bound and an asymptotic upper bound of the maximal
cardinality of $t$-deletion codes,
and he showed that VT codes achieve the upper bound in the single error case,
all in the same paper \cite{levenshtein1966binary}.
There is a gap between his upper bound and lower bound for multiple
deletion cases. To find a tight bound is still an open problem.

Construction of multi deletion codes is known to be challenging.
One successful construction is Helberg codes \cite{helberg2002multiple,abdel2012helberg,hagiwara2016short}.
A weakness of Helberg codes is its code size.
The size is far from both of Levenshtein's bounds, in particular for three or more deletions.
Recent surprising constructions were introduced by Sima et.al. \cite{sima2018two,sima2019optimal}.

The code size resulting from their construction approaches the lower bound,
while the size is still strictly smaller than the lower bound.

The authors' strategy to study the tightness of Levenshtein's upper bound
excludes explicit code construction.
This paper studies the tightness of the bound under the assumption of existence of perfect deletion codes.

In the 90's, Levenshtein also presented another remarkable work on VT codes:
VT codes were perfect for single deletion \cite{levenshtein1992perfect}.
Additionally he showed that
there existed perfect single deletion codes of length 3 for any alphabet size.
Bours extended this result, showing that
there existed perfect 2-deletion codes of length 4 for any alphabet size \cite{bours1995construction}.
Mahmoodi also extended the result, showing that
there existed perfect 3-deletion codes of length 5 for any alphabet size \cite{mahmoodi1998existence}.
It was also proven that there exist perfect 4-deletion codes
of length 6 for some specific alphabet size \cite{shalaby2002existence,yin2001combinatorial}.
Papers \cite{wang2006constructions,wang2008some,abel2010existence} are other examples of 
research for the existence of $(k-2)$-deletion codes of length $k$.

To determine the existence of perfect $t$-deletion codes of length $n$
over a binary alphabet is still an open and difficult problem.
This paper shows Levenshtein's upper bound for $t$-deletion codes is tight
for any $t$ and achievable by perfect $t$-deletion codes under the
assumption that there exist infinitely many perfect $t$-deletion binary codes.

\section{Notation and Preliminaries}\label{sec:2}
First, we introduce notation that is used throughout this paper.
The binary set $\{ 0, 1 \}$ is denoted by $\mathbb{B}$.
For a positive integer $n$, $\mathbb{B}^n$ denotes
the set of binary sequences of length $n$.
We define $\mathbb{B}^{0}$ as $\{ \epsilon \}$, where
$\epsilon$ is the empty word.
The symbol $t$ is used for the number of deletions.
Hence it is a positive integer.
For any set $X$, $\# X$ means the cardinality of $X$.

For a binary sequence $\mathbf{x}$, i.e., $\mathbf{x} \in \bigcup_{n \ge 0} \mathbb{B}^n$,
$|| \mathbf{x} ||$ denotes the number of runs of $\mathbf{x}$.
The set of sequences obtained by $t$-deletions to $\mathbf{x}$ is denoted by $\mathrm{dS}^{t} ( \mathbf{x} )$.
In this paper, $\mathrm{dS}^{t} ( \mathbf{x} )$ is called the \textit{deletion sphere}
\footnote{This is also called the \textit{deletion ball}.
In abstract mathematics, the terms sphere and ball are often used interchangeably.
}
 for $\mathbf{x}$.
The following are the upper and lower bounds of the cardinality of a deletion sphere.
\begin{fact}[\cite{levenshtein1966binary}]
\label{fact:levA}
$$
\binom{ || \mathbf{x} || - t + 1 }{ t }
\leq
\# \mathrm{dS}^{t} ( \mathbf{x} )
\leq
\binom{ || \mathbf{x} || + t - 1 }{ t }.
$$
\end{fact}

\begin{definition}[$t$-deletion codes]
Let $C$ be a set of binary sequences.
$C$ is called a \textbf{$t$-deletion code} if
for any distinct $\mathbf{c}$ and $\mathbf{d} \in C$,
$$
\mathrm{dS}^{t}(\mathbf{c}) \cap \mathrm{dS}^{t}(\mathbf{d}) = \emptyset.
$$
If the code $C$ is a subset of $\mathbb{B}^n$ for a positive integer $n$,
$C$ is called a \textit{$t$-deletion code of length} $n$.
\end{definition}

One of the open problems on deletion codes is
to determine the maximal cardinality of
$t$-deletion codes of length $n$ for given $t$ and $n$.
Let us set the notation:
 $$M^{t}(n) := \text{ the maximum cardinality of $t$-deletion codes of length $n$.} $$
Even if $t=1$, $M^{1}(n)$ is only known for $n \le 7$ \cite{sloane2002single}.

On the other hand,
its asymptotic upper and lower bounds have been obtained by Levenshtein.
Let us introduce notation $\sim$ and $\simeq$ for describing asymptotic bounds.
For functions $f$ and $g$ from the set $\mathbb{Z}_{> 0}$ of positive integers to the set $\mathbb{R}$ of real numbers,
$f(n) \sim g(n)$ means $\lim_{n \rightarrow \infty} f(n)/g(n) = 1$
and 
$f(n) \lesssim g(n) $ means $\lim_{n \rightarrow \infty} f(n)/g(n) \le 1$.
\begin{fact}[\cite{levenshtein1966binary}]
\label{fact:levB}
$$
\frac{ (t!)^2 2^{n+t}}{ n^{2t} } \lesssim M^t (n) \lesssim \frac{ t! 2^n }{ n^t }.
$$
\end{fact}

Levenshtein also showed that
the upper bound was achievable for $t=1$ by VT codes.
\begin{fact}[\cite{levenshtein1966binary}]
$$
\# \mathrm{VT}_n (0) \sim M^{t}(n) \sim \frac{2^n}{n},
$$
where $\mathrm{VT}_n (0)$ is a VT code defined below.
\end{fact}
\begin{definition}[VT codes]
For any positive integer $n$ and any integer $a$,
define
$$
\mathrm{VT}_{n} (a)
 := \{ \mathbf{x} \in \mathbb{B}^n \mid x_1 + 2 x_2 \dots + n x_n \equiv a \pmod{n+1} \}.
$$
The set $\mathrm{VT}_n (a)$ is called a \textbf{VT code}.
It is known that any VT code is a $1$-deletion code \cite{levenshtein1966binary}.
\end{definition}

An attractive property of VT codes is their perfectness.
To define perfect codes, the following notation $\mathrm{dS}^t (X)$ is introduced.
For a set $X$ of binary sequences, we define
$$\mathrm{dS}^{t} (X) := \bigcup_{\mathbf{x} \in X} \mathrm{dS}^{t} (x).$$
\begin{definition}[Perfect Codes]
Let $n \ge t$.
A $t$-deletion code $C$ of length $n$ is called a \textbf{perfect $t$-deletion code} if
$$
\mathrm{dS}^{t}(X) = \mathbb{B}^{n-t}.
$$
In other words, any short sequence $\mathbf{y} \in \mathbb{B}^{n-t}$
belongs to a deletion sphere $\mathrm{dS}^t (\mathbf{x})$ of some codeword $\mathbf{x} \in C$.
\end{definition}

\begin{example}
For any $\mathbf{x} \in \mathbb{B}^n$,
the singleton $\{ \mathbf{x} \}$ is trivially a perfect $n$-deletion code
since
$\mathrm{dS}^{n} ( \mathbf{x} ) = \{ \epsilon \} = \mathbb{B}^{n-n}$.

A remarkable non-trivial example is VT codes.
It is known that for any $a$ and $n$, $\mathrm{VT}_n (a)$ is a perfect $1$-deletion code
\cite{levenshtein1992perfect}.
\end{example}

The motivation of this research is the following question:
why is it that VT codes have two beautiful properties:
1. they achieve the asymptotic upper bound, and
2. they satisfy the perfectness.
Our main result provides an answer to this question:
\begin{theorem}\label{thm:main1}
Assume that there exists a perfect $t$-deletion code $C_n$ of length $n$ for each positive integer $n$.
Then
$$
\# C_n \sim M^t (n) \sim \frac{ t! 2^n }{ n^t }.
$$
\end{theorem}
The proof is given at the end of \S\ref{sec:4} with lemmas in \S\ref{sec:3} and \S\ref{sec:4}.

We conclude this section by introducing one more notation.
Given a positive integer $r$, $X_r$ denotes the subset of a set $X$ such that 
the number of run of the element is equal to $r$, i.e.,
$$X_{r} := \{ \mathbf{x} \in X \mid  || \mathbf{x} || = r \}.$$

\begin{fact}[\cite{levenshtein1966binary}]
\label{lemma:cardOfBnr}
For any positive integer $r$ with $1 \le r \le n$,
$$
\# (\mathbb{B}^n)_{r} = 2 \binom{n-1}{r-1}.
$$
\end{fact}

\section{Lemmas}\label{sec:3}
This section is devoted to lemmas that are useful to prove our main contribution,
i.e., Theorem \ref{thm:main1},
but that are provable without properties of deletion codes.
Most of the lemmas are statements on binomial coefficients.

\begin{lemma}
\label{lemma:doubleSumUbLb}
For any integers $0 \le I \le K$
and a function $f : \mathbb{Z} \rightarrow \mathbb{R}$,
if $f(x) \ge 0$ for any $x \in \mathbb{Z}$,
$$
(I + 1)
\sum_{I \le r \le K} 
f(r)
\le
\sum_{0 \le i \le I} \sum_{0 \le k \le K} f(i+k)
\le
(I+1) \sum_{0 \le r \le I+K}f(r)
\le
(K+1) \sum_{0 \le r \le I+K}f(r),
$$
where $\mathbb{Z}$ is the set of integers
and $\mathbb{R}$ is the set of real numbers.
\end{lemma}

\begin{proof}
There are just $I+1$ solutions on variables $0 \le i \le I$ and $0 \le k \le K$
of an equation $i+k = r$ for a given $I \le r \le K$.
Hence the leftmost inequality holds.

On the other hand,
there are at most $I+1$ solutions on variables $0 \le i \le I$ and $0 \le k \le K$
for an equation $i+k=r$ for a given $0 \le r \le I+K$.
Hence the second inequality holds.

The third one follows from the assumption that $I \le K$.
\end{proof}
\begin{lemma}
\label{lemma:binomExample3}
For any positive integers $n$ and $p$
with $p + t \le n$,
we have
$$
\frac{ \binom{n-t-1}{p-1} }
     { \binom{p+t-1}{t}   }
=
\frac{1}{\binom{n-1}{t}}
\binom{n-1}{p+t-1}.
$$
\end{lemma}

\begin{proof}
$$
\frac{ \binom{n-t-1}{p-1} }
     { \binom{p+t-1}{t}   }
=
\frac{ \frac{(n-t-1)!}{(p-1)!(n-t-p)!} }
     { \frac{(p+t-1)!}{t! (p-1)!} }
=
\frac{ t! (n-t-1)! }{ (p+t-1)! (n-t-p)! }.
$$

On the other hand,
$$
\frac{1}{\binom{n-1}{t}}
\binom{n-1}{p+t-1}
=
\frac{ \frac{(n-1)!}{(p+t-1)!(n-p-t)!} }
     { \frac{(n-1)!}{t! (n-t-1)!} }
=
\frac{ t! (n-t-1)! }{ (p+t-1)! (n-t-p)! }.
$$
\end{proof}

\begin{lemma}
\label{lemma:binomBoundExample2}
For any positive integers $a, b$, and $c$ with $c \le a \le b$,
we have
$$
\frac{ \binom{c+t-1}{t} }
{ \binom{c+b-a+3t-1}{t} }
\frac{ \binom{b+3t-1}{t} }
{ \binom{a+t-1}{t} }
\le 1.
$$
\end{lemma}
\begin{proof}
\begin{align*}
\frac{ \binom{c+t-1}{t} }
{ \binom{c+b-a+3t-1}{t} }
\frac{ \binom{b+3t-1}{t} }
{ \binom{a+t-1}{t} }
&=
\frac{ \frac{ (c+t-1)!}     { t! (c-1)!}        }
     { \frac{ (c+b-a+3t-1)!}{ t! (c+b-a+2t-1)!} }
\frac{ \frac{ (b+3t-1)!}    { t! (b+2t-1)! }    }
     { \frac{ (a+t-1)!}     { t! (a-1)!}        }\\
&=
\frac{ \frac{ (c+t-1)!}     { (c-1)!}        }
     { \frac{ (c+b-a+3t-1)!}{ (c+b-a+2t-1)!} }
\frac{ \frac{ (b+3t-1)!}    { (b+2t-1)! }    }
     { \frac{ (a+t-1)!}     { (a-1)!}        }\\
&=
\frac{ \prod_{0 \le i \le t-1} c+i}
     { \prod_{0 \le i \le t-1} c+b-a+2t+i }
\frac{ \prod_{0 \le i \le t-1} b+2t + i }
     { \prod_{0 \le i \le t-1} a+i }\\
&= \prod_{0 \le i \le t-1} \frac{ (c+i)(b+2t+i) }{ (c+b-a+2t+i)(a+i)}.
\end{align*}
Hence it is enough to show 
$$\frac{ (c+i)(b+2t+i) }{ (c+b-a+2t+i)(a+i)} \le 1$$
for each $0 \le i \le t-1$.
By routine calculation, we can verify that this is equivalent to
$$
0 \le (a-c)(b - a + 2t).
$$
This follows from the assumption $c \le a \le b$.
\end{proof}

\begin{lemma}
\label{lemma:convergeMinority}
$$
\lim_{n \rightarrow \infty}
 \sum_{0 \le r \le 2 \lfloor n^{2/3} \rfloor}
  \frac{\binom{n-1}{r}}
       { 2^{n-1} } = 0.
$$
\end{lemma}
\begin{proof}
To prove the statement,
we can assume that $n$ is sufficiently large.
This assumption implies $2 \lfloor n^{2/3} \rfloor < n$
and $(2 n^{2/3} + 2/3) \log_2 n < n/2$.

For any $0 \le r \le 2 \lfloor n^{2/3} \rfloor $,
\begin{align*}
\binom{n-1}{r}
&=   \frac{ (n-1)(n-2) \dots (n-r) }{ r! }\\
&\le (n-1)(n-2) \dots (n-r)\\
&\le n^r \\
&\le n^{2 \lfloor n^{2/3} \rfloor}.
\end{align*}

Hence
\begin{align*}
\sum_{0 \le r \le 2 \lfloor n^{2/3} \rfloor}
 \binom{n-1}{r}
 &\le
  \sum_{0 \le r \le 2 \lfloor n^{2/3} \rfloor}
   n^{2 \lfloor n^{2/3} \rfloor}\\
 &\le
   (2 \lfloor n^{2/3} \rfloor + 1)
   n^{2 \lfloor n^{2/3} \rfloor}\\
 &\le
   3 \lfloor n^{2/3} \rfloor
   n^{2 \lfloor n^{2/3} \rfloor}\\
 &\le
   3  n^{2/3} 
   n^{2 n^{2/3} }\\
 &=
   3 n^{ 2n^{2/3} + 2/3 }\\
 &=
   3 \cdot 2^{ (2n^{2/3} + 2/3) \log_2 n }.
\end{align*}

Therefore
\begin{align*}
 \sum_{0 \le r \le 2 \lfloor n^{2/3} \rfloor}
  \frac{\binom{n-1}{r}}
       { 2^{n-1} }
&\le
   3 \cdot 2^{ (2n^{2/3} + 2/3) \log_2 n }
   / 2^{n-1}\\
&=
   6 \cdot 2^{ (2n^{2/3} + 2/3) \log_2 n - n }.
\end{align*}

To conclude the proof, it is enough to show
$$
\lim_{n \rightarrow \infty}
 \left( (2n^{2/3} + 2/3) \log_2 n - n \right)
  = - \infty.
$$
This is obtained from the following inequality
$$
(2n^{2/3} + 2/3) \log_2 n - n < n/2 - n = -n/2.
$$
\end{proof}

\begin{lemma}
\label{lemma:FourConverges}
For any non-negative integer $t$,
$$
\lim_{n \rightarrow \infty}
\frac{\lfloor n^{1/3} \rfloor + 1}{\lfloor n^{1/3} \rfloor + 1 + 2t}
=
1,
$$
$$
\lim_{n \rightarrow \infty}
\frac{ \binom{\lfloor n^{2/3} \rfloor + t - 1}{t} } { \binom{ \lfloor n^{2/3} \rfloor + \lfloor n^{1/3} \rfloor + 3t - 1}{t} }
=
1,
$$
$$
\lim_{n \rightarrow \infty}
\sum_{r < \lfloor n^{1/3} \rfloor - 1 + \lfloor n^{2/3} \rfloor + t} \frac{\binom{n - 1}{ r}}{2^{n-1}}
=
0,
$$
and
$$
\lim_{n \rightarrow \infty}
\sum_{r > n - 1 - \lfloor n^{2/3} \rfloor + t} \frac{\binom{n - 1}{ r}}{2^{n-1}}
=
0.
$$
\end{lemma}
\begin{proof}
The first two equations follow from fundamental calculus.

For the third equation,
we claim
$$
0 \le \text{L.H.S.}
 \le
  \lim_{n \rightarrow \infty} 
   \sum_{0 \le r \le 2 \lfloor n^{2/3} \rfloor}
    \frac{\binom{n-1}{r}}
         {2^{n-1}},
$$
since
$$
 \lfloor n^{1/3} \rfloor - 1 + \lfloor n^{2/3} \rfloor + t
  <
   2 \lfloor n^{2/3} \rfloor
$$
for sufficiently large $n$.
By using Lemma \ref{lemma:convergeMinority}, we obtain the third equation.

For the last equation,
we remark that
\begin{align*}
0
 &\le \sum_{r > n - 1 - \lfloor n^{2/3} \rfloor + t} \frac{\binom{n - 1}{ r}}{2^{n-1}}\\
 &=
  \sum_{r < \lfloor n^{2/3} \rfloor - t} \frac{\binom{n - 1}{ r}}{2^{n-1}}\\
 &\le
  \sum_{0 \le r < 2 \lfloor n^{2/3} \rfloor} \frac{\binom{n - 1}{ r}}{2^{n-1}}.
\end{align*}
By using Lemma \ref{lemma:convergeMinority}, we obtain the third equation.
\end{proof}

\begin{lemma}\label{lemma:powerToBinom}
$$ \frac{ t! 2^n }{ n^t } \sim \frac{ 2^n }{ \binom{n-1}{t} }.$$
\end{lemma}
\begin{proof}
\begin{align*}
\frac{ t! 2^n }{ n^t } \sim \frac{ 2^n }{ \binom{n-1}{t} }
&\iff
 \frac{ t! }{ n^t } \sim \frac{1 }{ \binom{n-1}{t} } = \frac{t!}{n (n-1) \dots (n-t+1)} \\
&\iff
 n^t \sim n (n-1) \dots (n-t+1) \\
&\iff
\lim_{n \rightarrow \infty}
 \frac{n (n-1) \dots (n-t+1)}{n^t} = 1.
\end{align*}
It is easy to show the last equality, \textit{i.e.},
$\lim_{n \rightarrow \infty} \frac{n (n-1) \dots (n-t+1)}{n^t} = 1$.
\end{proof}

Note that the function $l$ in the next lemma will appear again
as a lower bound of the cardinality of any perfect $t$-deletion codes of length $n$ with $t < n$.

\begin{lemma}
\label{lemma:lowerBoundAsymp}
Define a function $l : \mathbb{Z}_{>0} \rightarrow \mathbb{R}$ by setting 
\begin{align*}
l(n) &:=
\frac{2^n}{\binom{n-1}{t}}
\frac{\lfloor n^{1/3} \rfloor + 1}{\lfloor n^{1/3} \rfloor + 1 + 2t}
\frac{ \binom{\lfloor n^{2/3} \rfloor + t - 1}{t} } { \binom{ \lfloor n^{2/3} \rfloor + \lfloor n^{1/3} \rfloor + 3t - 1}{t} }\\
& \times
\left(
 1 
  - \sum_{r < \lfloor n^{1/3} \rfloor - 1 + \lfloor n^{2/3} \rfloor + t} \frac{\binom{n - 1}{ r}}{2^{n-1}}
  - \sum_{r > n - 1 - \lfloor n^{2/3} \rfloor + t} \frac{\binom{n - 1}{ r}}{2^{n-1}}
\right).
\end{align*}
Then
$$ l(n) \sim \frac{t! 2^n}{n^t}.$$
\end{lemma}
\begin{proof}
This immediately follows from
Lemma \ref{lemma:FourConverges}
and Lemma \ref{lemma:powerToBinom}.
\end{proof}

\section{Analysis on the Cardinality for Perfect Deletion Codes}\label{sec:4}
\begin{lemma}
\label{lemma:UB_cardOfDelSph}
For any positive integers $a$ and $b$
with $a \le b$,
$$
\bigcup_{a - 2t \le r \le b} (\mathbb{B}^{n-t})_{r}
\supset
\mathrm{dS}^t ( \bigcup_{a \le r \le b} \mathbb{B}^n_{r} ).
$$
In particular,
for any $\mathbf{c} \in \mathbb{B}^n$ with $|| \mathbf{c} || = r$
and for any $\mathbf{y} \in \mathrm{dS}^{t}( \mathbf{c} )$,
the inequalities $|| \mathbf{y} || \le || \mathbf{c} || \le || \mathbf{y} || + 2t$ hold.
\end{lemma}
\begin{proof}
It is enough to show that
for any $\mathbf{c} \in (\mathbb{B}^n)_a$,
$$
\bigcup_{a - 2t \le r \le a} ( \mathbb{B}^{n-t} )_r
\supset
\mathrm{dS}^{t} (\mathbf{c}).
$$

By a single deletion, the number of runs cannot increase
and can only decrease by at most two,
i.e., 
$a - 2 \le || \mathbf{y} || \le a$.
Hence, by $t$-deletions, the number of runs is between $a-2t$ and $a$.
Additionally, the length is just $n-t$.
\end{proof}

From here, we focus on the case where $C$ is a perfect $t$-deletion code of length $n$
with $t < n$.
\begin{lemma}
\label{lemma:LB_cardOfPerfect}
Let $C$ be a perfect $t$-deletion code of length $n$ with $t < n$.
For any positive integers $a$ and $b$ with $a \le b \le n - t$,
$$
\mathrm{dS}^t ( \bigcup_{a \le r \le b + 2t} C_{r} )
\supset
\bigcup_{a \le r \le b }
(\mathbb{B}^{n-t})_{r}.
$$
Hence
$$
\sum_{a \le r \le b + 2t} \# \mathrm{dS}^t ( C_{r} )
\ge
2 \sum_{a \le p \le b} \binom{n-t-1}{p-1}.
$$
\end{lemma}
\begin{proof}
By the assumption of perfectness for $C$,
we have
$$
\mathbb{B}^{n-t} = \mathrm{dS}^{t} (C).
$$

Hence for any $\mathbf{y} \in \mathbb{B}^{n-t}$,
there exists $c \in C$ such that 
\begin{align}
\mathbf{y} &\in \mathrm{dS}^{t} (c). \label{step1:inProof_LB_cardOfPerfect}
\end{align}
By Lemma \ref{lemma:UB_cardOfDelSph},
$ || c || - 2t \le || y || \le || c ||$ holds.
In other words,
$$ || y || \le || c || \le || y || + 2t.$$
Hence 
\begin{align}
 c \in \bigcup_{ || x || \le r \le || x || + 2t } C_r. \label{step2:inProof_LB_cardOfPerfect}
\end{align}

By (\ref{step1:inProof_LB_cardOfPerfect}) and (\ref{step2:inProof_LB_cardOfPerfect}),
$$
\mathbf{y} \in \mathrm{dS}^t ( \bigcup_{ || y || \le r \le || y || + 2t } C_r ).
$$

If $\mathbf{y}$ is an element of $\bigcup_{a \le r \le b} (\mathbb{B}^{n})_r$,
i.e., $a \le || \mathbf{y} || \le b$,
the statement is obtained.

On the other hand,
the inequality statement follows from
$ \mathrm{dS}^t ( \cup_{a \le r \le b + 2t} C_r ) = \cup_{a \le r \le b + 2t} \mathrm{dS}^t ( C_r )$
and from Fact \ref{lemma:cardOfBnr}.
\end{proof}

Conversely, the following is the upper bound of 
$\sum_{a \le r \le b + 2t} \# \mathrm{dS}^t ( C_{r} )$.
\begin{lemma}
\label{lemma:UB_cardOfDelSphere}
Let $C$ be a perfect $t$-deletion code of length $n$.
For any positive integers $a$ and $b$ with $a \le b \le n$,
$$
\binom{b + 3t - 1}{t} \sum_{a \le r \le b + 2t} \# C_r
\ge
\sum_{a \le r \le b + 2t} \# \mathrm{dS}^t ( C_{r} ).
$$
\end{lemma}
\begin{proof}
\begin{align*}
\sum_{a \le r \le b + 2t} \# \mathrm{dS}^t ( C_{r} )
&=
 \sum_{a \le r \le b + 2t} \sum_{\mathbf{x} \in C_r} \# \mathrm{dS}^t ( \mathbf{x} )
  & (\text{$C$ is a $t$-deletion code.})\\
&\le
 \sum_{a \le r \le b + 2t} \sum_{\mathbf{x} \in C_r} \binom{ || \mathbf{x} || + t - 1}{t}
  & (\text{Fact \ref{fact:levA}.})\\
&=
 \sum_{a \le r \le b + 2t} \sum_{\mathbf{x} \in C_r} \binom{ r + t - 1}{t}
  & (\mathbf{x} \in C_r.)\\
&=
 \sum_{a \le r \le b + 2t} \# C_r \binom{ r + t - 1}{t}
  & (\sum_{\mathbf{x} \in C_r} 1 = \# C_r.)\\
&\le
 \binom{b + 2t + t - 1}{t}  \sum_{a \le r \le b + 2t} \# C_r.
  & (r \le b + 2t.)
\end{align*}
\end{proof}

\begin{lemma}
\label{lemma:LB_cardOfPerfect2}
Let $C$ be a perfect $t$-deletion code of length $n$.
For any positive integers $a, b,$ and $c$ with $c \le a < b \le n$,
$$
\sum_{a \le r \le b + 2t} \# C_r
\ge
\frac{ \binom{c + t -1}{t} }{ \binom{c + b - a + 3t - 1}{t} }
\frac{2}{ \binom{n-1}{t} }
 \sum_{a \le p \le b}
  \binom{n - 1}{p + t - 1}.
$$
\end{lemma}
\begin{proof}
Let us start with combining 
Lemmas \ref{lemma:LB_cardOfPerfect} and \ref{lemma:UB_cardOfDelSphere}:
\begin{align}
\binom{b + 3t - 1}{t} \sum_{a \le r \le b + 2t} \# C_r
\ge
2 \sum_{a \le p \le b} \binom{n - t - 1}{p - 1}.   \label{eq:proof_LB_cardOfPerfect2}
\end{align}

For applying Lemma \ref{lemma:binomBoundExample2},
and the inequality (\ref{eq:proof_LB_cardOfPerfect2}),
the following holds:
$$
\sum_{a \le r \le b + 2t} \# C_r
\ge
\frac{ \binom{c + t -1}{t} }{ \binom{c + b - a + 3t - 1}{t} }
 \frac{ \binom{b + 3t - 1}{t} } { \binom{a + t -1}{t} }
\sum_{a \le r \le b + 2t} \# C_r
\ge
\frac{ \binom{c + t -1}{t} }{ \binom{c + b - a + 3t - 1}{t} }
 \frac{ 1 }{ \binom{a + t -1}{t} }
 2
 \sum_{a \le p \le b} \binom{n - t - 1}{p - 1}.
$$

Furthermore,
\begin{align*}
\text{R.H.S.}
&=
\frac{ \binom{c + t -1}{t} }{ \binom{c + b - a + 3t - 1}{t} }
 2
 \sum_{a \le p \le b}
 \frac{ \binom{n - t - 1}{p - 1} } { \binom{a + t -1}{t} }\\
&\ge
\frac{ \binom{c + t -1}{t} }{ \binom{c + b - a + 3t - 1}{t} }
2
 \sum_{a \le p \le b}
 \frac{ \binom{n - t - 1}{p - 1} }{ \binom{p + t -1}{t} }
 & ( a \le p \text{ implies } \binom{a+t-1}{t} \le \binom{p+t-1}{t}. ) \\
&=
\frac{ \binom{c + t -1}{t} }{ \binom{c + b - a + 3t - 1}{t} }
\frac{2}{ \binom{n-1}{t} }
 \sum_{a \le p \le b}
  \binom{n - 1}{p + t - 1}.
  & (\text{by Lemma \ref{lemma:binomExample3}}.)
\end{align*}
\end{proof}

The following statement provides a non-asymptotic lower bound for perfect $t$-deletion codes.
\begin{theorem}\label{thm:main2}
Let $C$ be a perfect $t$-deletion code of length $n$ with $n > t$.
Then
\begin{align*}
\# C
& \ge
\frac{2^n}{\binom{n-1}{t}}
\frac{\lfloor n^{1/3} \rfloor + 1}{\lfloor n^{1/3} \rfloor + 1 + 2t}
\frac{ \binom{\lfloor n^{2/3} \rfloor + t - 1}{t} } { \binom{ \lfloor n^{2/3} \rfloor + \lfloor n^{1/3} \rfloor + 3t - 1}{t} }\\
& \times 
\left(
 1 
  - \sum_{r < \lfloor n^{1/3} \rfloor - 1 + \lfloor n^{2/3} \rfloor + t} \frac{\binom{n - 1}{ r}}{2^{n-1}}
  - \sum_{r > n - 1 - \lfloor n^{2/3} \rfloor + t} \frac{\binom{n - 1}{ r}}{2^{n-1}}
\right).
\end{align*}
\end{theorem}
\begin{proof}
For any non-negative integer $i$ and any positive integers $j$ and $c$ with $i + j \le n-c$,
the following inequalities hold $$ c \le c+i < c + i + j \le n.$$
Set $a := c+i$ and $b := c + i + j$ into Lemma \ref{lemma:LB_cardOfPerfect2}:
\begin{align}
\sum_{c+i \le r \le c+i+j+2t } \# C_r
&\ge
\frac{ \binom{c+t-1}{t} }{ \binom{c+j+3t-1}{t} } \frac{2}{ \binom{n-1}{t} }
\sum_{c + i \le p \le c + i + j} \binom{ n-1 }{ p+t-1 }.
\label{eq:thmCard.1}
\end{align}

Then consider the sum of inequalities for $i$ over $\{0, 1, \dots, n-2c \}$.
The left hand side of the sum of (\ref{eq:thmCard.1}) is upper bounded as below
\begin{align*}
\text{L.H.S.}
&=
\sum_{0 \le i \le n - 2c}
\sum_{c+i \le r \le c+i+j+2t } \# C_r\\
&=
\sum_{0 \le i \le n - 2c}
\sum_{0 \le r \le j+2t }  \# C_{c+i + r}\\
&\le
(j + 2t + 1)
\sum_{0 \le i \le n - 2c + j+2t}
\# C_{c+i} & \text{(by Lemma \ref{lemma:doubleSumUbLb})}\\
&\le
 (j+2t+1) \# C. & \text{(by $\#C = \sum_{1 \le r \le n} \# C_r$.)}
\end{align*}

On the other hand,
the right hand side of the sum of (\ref{eq:thmCard.1}) is lower bounded as below
\begin{align*}
\text{R.H.S.}
&=
\sum_{0 \le i \le n - 2c}
\frac{ \binom{c+t-1}{t} }{ \binom{c+j+3t-1}{t} } \frac{2}{ \binom{n-1}{t} }
\sum_{c + i \le p \le c + i + j} \binom{ n-1 }{ p+t-1 }\\
&=
\frac{ \binom{c+t-1}{t} }{ \binom{c+j+3t-1}{t} } \frac{2}{ \binom{n-1}{t} }
\sum_{0 \le i \le n - 2c}
\sum_{c + i \le p \le c + i + j} \binom{ n-1 }{ p+t-1 }\\
&=
\frac{ \binom{c+t-1}{t} }{ \binom{c+j+3t-1}{t} } \frac{2}{ \binom{n-1}{t} }
\sum_{0 \le i \le n - 2c}
\sum_{0 \le p \le j} \binom{ n-1 }{ c+i+p+t-1 }\\
&\ge
\frac{ \binom{c+t-1}{t} }{ \binom{c+j+3t-1}{t} } \frac{2}{ \binom{n-1}{t} }
(j+1)
\sum_{j \le r \le n - 2c}
\binom{ n-1 }{ c+r+t-1 }  & \text{(by Lemma \ref{lemma:doubleSumUbLb})}\\
&=
\frac{ \binom{c+t-1}{t} }{ \binom{c+j+3t-1}{t} } \frac{2}{ \binom{n-1}{t} }
(j+1)
\sum_{c+j \le r \le n - c}
\binom{ n-1 }{ r+t-1 }.
& (c+r \mapsto r)
\end{align*}

By combining both the upper bound and the lower bound above,
we have
\begin{align*}
\# C
&\ge
 \frac{j+1}{j+2t+1}
 \frac{ \binom{c+t-1}{t} }{ \binom{c+j+3t-1}{t} }
 \frac{ 2 }{ \binom{ n-1 }{t} }
 \sum_{c + j \le r \le n-c} \binom{n-1}{t-1+r}\\
&=
 \frac{j+1}{j+1+2t}
 \frac{ \binom{c+t-1}{t} }{ \binom{c+j+3t-1}{t} }
 \frac{ 2^n }{ \binom{ n-1 }{t} }
 \left(
 1
 - \sum_{r < t-1+c+j} \frac{ \binom{n-1}{r} }{2^{n-1}}
 - \sum_{r > t-1+n-c} \frac{ \binom{n-1}{r} }{2^{n-1}}
 \right).
\end{align*} 
The last equation follows from
$$
\sum_{c+j \le r \le n-c} \binom{n-1}{t-1+r}
=
2^{n-1}
 - \sum_{r < t-1+c+j} \binom{n-1}{r}
 - \sum_{r > t-1+n-c} \binom{n-1}{r}.
$$

By setting $j := \lfloor n^{1/3} \rfloor$ and $c := \lfloor n^{2/3} \rfloor$,
we conclude this proof.
\end{proof}

Finally we are ready to prove the main contribution.
\begin{proof}[Proof of Theorem \ref{thm:main1}]
Note that the lower bound of 
Theorem \ref{thm:main2} is as the same as $l(n)$ in Lemma \ref{lemma:lowerBoundAsymp}.
Hence
$$
\frac{t! 2^n}{n^t} \sim l(n) \lesssim C_n.
$$
On the other hand,
$\# C_n \le M^t (n)$ holds in general,
and $M^t(n) \lesssim \frac{t! 2^n}{n^t}$ holds by Fact \ref{fact:levB}.
Therefore
$$
C_n \lesssim M^t (n) \lesssim \frac{t! 2^n}{n^t}.
$$
This implies that the statement holds.
\end{proof}

\section{Conclusion and Remarks}\label{sec:5}

This paper focused on the cardinality of perfect multi deletion binary codes.
The lower bound for any perfect $t$-deletion code
and
the asymptotic achievable of Levenshtein's upper bound 
were obtained as Theorems \ref{thm:main2} and \ref{thm:main1} respectively.
The lemmas obtained in \S4 are useful for analysis on the cardinality of perfect deletion codes.
The authors would like to leave following remarks.
\begin{remark}
We can relax the assumption in Theorem \ref{thm:main2} that 
there exists a perfect $t$-deletion code of length $n$
for ``each positive integer $n$.''
For example, an assumption that
there exists infinitely many perfect $t$-deletion codes
is enough to achieve the asymptotic optimality.
Let us assume there exists infinitely many such codes.
Let $A$ denote the set of code lengths for which $t$-deletion codes exists.
Let $A = \{ a_1, a_2, a_3, \dots \}$ with $a_i < a_{i+1}$ for any $i \ge 1$.
Then we can rewrite the statement
$$
C_{a_n} \sim M^t (a_n) \sim \frac{t! 2^{a_n}}{(a_n)^t}.
$$
\end{remark}
\begin{remark}
By an exhaustive computer experimental search,
the authors have verified that there exist perfect 2-deletion code of length $2,3,4,5$ and $6$.
The following are examples of $2$-deletion codes.
\begin{itemize}
\item Any singleton in $\mathbb{B}^2$ is a perfect code of length $2$.
There are just 4 perfect codes of length 2.
\item The code $\{ 000, 111 \}$ is a perfect code of length $3$.
There are just 7 perfect codes of length 3.
\item The code $\{ 0100, 1111 \}$ is a perfect code of length $4$.
There are just 10 perfect codes of length 4.
\item The code $\{ 01010, 11111\}$ is a perfect code of length $5$.
There are just 12 perfect codes of length 5.
\item The code $\{ 000000, 111000, 010101, 111111 \}$ is a perfect code of length $6$.
There are just 10 perfect codes of length 6.
\end{itemize}

On the other hand, 
the authors also have verified that
there do not exist any perfect $2$-deletion codes of length $7$.
The existence for longer lengths has not been checked yet.

It is also interesting to extend our results over other deletion channels,
e.g.,
burst deletions \cite{levenshtein1967asymptotically,schoeny2017codes},
balanced adjacent deletions \cite{hagiwara2017perfect},
channel models for racetrack memory \cite{chee2017coding,sima2019racetrack},
bounded insertions/deletions \cite{nozaki2019bounded}
and so on.

\end{remark}
\begin{remark}
Conversely, if we obtain a proof of that the upper bound is not tight for some $t$,
it implies that only finite perfect $t$-deletion codes exist.
\end{remark}

\section*{Acknowledgment}
\addcontentsline{toc}{section}{Acknowledgment}

The research has been partly executed in response
to support of KIOXIA Corporation (former Toshiba Memory Corporation) 
and JSPS KAKENHI Grant Number 18H01435, Grant-in-Aid for Scientific Research (B).

\bibliographystyle{plain}
\bibliography{reference}
\end{document}